\documentclass[journal]{aiaa-pretty}    

\usepackage{amssymb}
\usepackage{amsmath}
\usepackage{booktabs}
 \usepackage{varioref}
 \usepackage{wrapfig}
 \usepackage{threeparttable}
 \usepackage{dcolumn}
  \newcolumntype{d}{D{.}{.}{-1}}
 \usepackage{nomencl}
  \makeglossary
 \usepackage{subfigure}
 \usepackage{graphicx}
 \usepackage{subfigmat}
 \usepackage{fancyvrb}
  \fvset{fontsize=\footnotesize,xleftmargin=2em}
 \usepackage{lettrine}

\newtheorem{theorem}{Theorem}
\newtheorem{proposition}{Proposition}
\newtheorem{lemma}{Lemma}
\newtheorem{remark}{Remark}
\newtheorem{definition}{Definition}
\newenvironment{proof}{{\it Proof. }}{\hfill $\Box$}

\newcommand{\Real}{\mathbb R}

\newcommand{\real}[1]{{\mathbb R}^{#1}}
\newcommand{\bff}{{\boldsymbol f}}

\newcommand{\bx}{\boldsymbol x}

\newcommand{\bA}{{\boldsymbol A}}

\newcommand{\bpsi}{{\mbox{\boldmath $\psi$}}}

\author{ %
I. M. Ross\thanks{Distinguished Professor and Program Director, Control and Optimization 
} \\
\textit{Naval Postgraduate School, Monterey, CA 93943}
}
\title{A Direct Shooting Method is Equivalent to an Indirect  Method}

\abstract{
We show that a direct shooting method is mathematically equivalent to an indirect method in the sense of certain first-order conditions.  Specific mathematical formulas pertaining to the equivalence of a direct shooting method with an indirect method are derived. We also show that a theoretical equivalence does not necessarily translate to practical equivalence if the parameterized optimal control problem is simply patched to a nonlinear programming solver. A mathematical explanation is provided for the successes and failures of such patched nonlinear programming methods. In order to generate the correct solution more consistently, the nonlinear programming solver used in a traditional direct method must be replaced or augmented by a Hamiltonian programming method.  The theoretical results derived in this paper further strengthen the connections between computational optimal control, deep learning and automatic differentiation.
}

\begin{document}
\maketitle

\section{Introduction}

Starting from first principles, we show that a direct shooting method is mathematically equivalent to an indirect method. The mathematical equivalence is established in the sense of certain first-order conditions.  Given this equivalence, we revisit some widely-held beliefs such as the folklore that an ``indirect method is more accurate than a direct method.''  In fact, we show that the apparent differences between direct and indirect methods can be mathematically reconciled by realizing that a direct method masks its indirect elements.  By unmasking the indirect elements of a direct shooting method, we present new results on convergence and accuracy.
The new results provide mathematical formulas that help explain the historical direct-indirect chasm while simultaneously generating several prescriptions for closing the apparent gap.

The main contribution of this paper is a recalibration of certain entrenched concepts in trajectory optimization methods.  In particular, our theoretical results provide a mathematical explanation for the successes and failures of methods that patch discretized optimal control problems to generic nonlinear programming techniques.  Far superior performance capabilities can he harnessed if the generic nonlinear programming methods are replaced by mathematical programming techniques that are centered around the covector mapping principle\cite{ross-book}.  A first-principles' approach combined with a new theory on optimization\cite{rossJCAM-1} is used to present the main ideas.

\section{A Signature Problem For Analysis}
We consider the following ``signature'' trajectory optimization problem,
\begin{eqnarray}
& (\textsf{$P$}) \left\{
\begin{array}{lrl}
\emph{Minimize } & J[x(\cdot), u(\cdot)] :=& E\big(x(t_f)\big)\\
\emph{Subject to}& \dot x(t) =& f(x(t), u(t))  \\
& (x_0, t_0, t_f)  = & (x^0, t^0, t^f)
\end{array} \right. & \label{eq:probP}
\end{eqnarray}
It turns out that this problem is also a fundamental problem in deep learning\cite{DL=OCP}.
In Problem $(P)$, the decision variables are the state and control trajectories denoted by $x(\cdot)$ and $u(\cdot)$ respectively.  The pair $[x(\cdot), u(\cdot)]$ is the system trajectory. We assume one-dimensional state and control variables; i.e., $x \in \Real$ and $u \in \Real$. The one-dimensional assumption allows us to eliminate the unnecessary mathematics of cumbersome bookkeeping associated with vectors.  While such bookkeeping is indeed necessary for practical problems, it is quite unnecessary to understand the basic principles of trajectory optimization methods\cite{birk-TN}.

Given a system trajectory, the cost functional $J$ in Problem $(P)$ is computed via the endpoint cost function $E: x(t_f) \mapsto \Real$.  For a system trajectory to be feasible, it must satisfy the nonlinear dynamical equations $\dot x(t) = f(x(t), u(t))$ for all time $ t \in [t^0, t^f]$, where $t^0$ and $t^f$ are specified clock times.  The initial state $x(t_0)$ is fixed at $x^0$ while the final state $x(t_f)$ is free. The problem is then to find a system trajectory that minimizes the cost functional.

Suppose a candidate optimal trajectory is computed by a direct method. If this trajectory is claimed to be optimal, it must satisfy the necessary conditions for optimality.  \textbf{\emph{This argument is worth emphasizing because necessary conditions are not optional; i.e., they are indeed necessary, a point that is sometimes lost when optimality claims are made on solutions computed by classical direct methods}}.

The necessary conditions for Problem $(P$) are easily stated in terms of the Pontryagin Hamiltonian\cite{ross-book} defined by,
\begin{equation}\label{eq:H:=}
H(\lambda, x, u):= \lambda f(x,u)
\end{equation}
where, $\lambda \in \Real$ is a costate variable that satisfies the adjoint equation,
\begin{equation}
-\dot\lambda(t) = \partial_x H(\lambda(t), x(t), u(t)) = \lambda(t) \partial_x f(x(t), u(t))
\end{equation}
and the transversality condition,
\begin{equation}
\lambda(t_f) = \partial_x E(x(t_f))
\end{equation}
By assumption $u$ is unconstrained; hence, at each point along a candidate optimal solution, the first-order Hamiltonian minimization condition,
\begin{equation}\label{eq:HMC-FO}
\partial_u H(\lambda(t), x(t), u(t)) = 0
\end{equation}
must be satisfied.  Collecting all relevant equations, we can state the following: If a system trajectory is claimed to be optimal, it must satisfy the following first-order necessary conditions,
\begin{eqnarray}
& (\textsf{$P^\lambda$}) \left\{
\begin{array}{lrl}
& \dot x(t) =& f(x(t), u(t))  \\
&-\dot\lambda(t) = & \lambda(t)\partial_x f(x(t), u(t))\\
& \partial_u H(\lambda(t), x(t), u(t)) = & 0 \\
& (x_0, t_0, t_f)  = & (x^0, t^0, t^f)\\
&\lambda(t_f) = &\partial_x E(x(t_f))
\end{array} \right. & \label{eq:probPpsi}
\end{eqnarray}
\begin{remark}
Strictly speaking, $(P^\lambda)$ is not a definition of a problem; it is simply a statement of the necessary conditions.  If  $(P^\lambda)$ is reframed in terms of finding a costate trajectory $\lambda(\cdot)$ in addition to finding a system trajectory, then it may be viewed as defining a differential-algebraic boundary value problem.
\end{remark}
Conventionally, a direct method is framed in terms of solving $(P)$ and an indirect method as solving $(P^\lambda)$ \cite{vonStryk-survey,betts-survey}.  To reexamine this statement more carefully, we first develop discretizations of $(P)$ and $(P^\lambda)$.

\section{Basic Discretizations for Optimization Analysis}
We use forward/backward Euler discretizations to reveal certain fundamental outcomes in optimization.  It will be apparent shortly that the main aspects of our analysis is portable across alternative discretizations (e.g., higher-order Runge-Kutta methods) because our use of Euler methods is centered at a fundamental mathematical level. Euler-specific aspects are noted so that a reader can independently gauge the generality of specific results and/or statements. To this end, consider a uniform grid of ``diameter'' $h$ over the time interval $[t^0, t^f]$ defined by
\begin{equation}\label{eq:grid}
h := \frac{t^f - t^0}{N},  \quad N \in \mathbb{N}^+
\end{equation}
Time is indexed according to $t^0 = t_0 < t_1:= t_0 + h < t_2 := t_1+ h < \cdots < t_N = t^f$.
Discretized versions of continuous-time variables over this grid will be denoted by the same indexing scheme. For instance, a discretized state trajectory over \eqref{eq:grid} will be denoted as $ (x_0, x_1, \ldots, x_{N-1}, x_N)  $.
Replacing $\dot x(t)$ by its forward difference approximation results in the following algebraic equations as candidate approximations for the dynamics:
\begin{equation}\label{eq:FE4xdot=f}
\begin{aligned}
x_1 &= x^0 + h f(x^0, u_0) \\
 & \vdots \\
x_{N-1} & = x_{N-2} + h f(x_{N-2}, u_{N-2})\\
x_N & = x_{N-1} + h f(x_{N-1}, u_{N-1})
\end{aligned}
\end{equation}
Thus, a forward Euler discretization of Problem $(P)$ may be expressed as,
\begin{eqnarray}
& (\textsf{$P^{N}$}) \left\{
\begin{array}{lll}
\emph{Minimize } & J^N(X^N, U^N) := E(x_N)\\
\emph{Subject to}& x_1 = x^0 + h f(x^0, u_0) \\
& x_2  = x_1 + h f(x_1, u_1) \\
&\quad  \vdots \\
&x_N  = x_{N-1} + h f(x_{N-1}, u_{N-1})
\end{array} \right. & \label{eq:probPN}
\end{eqnarray}
where, $X^N \in \real{N}$ and $U^N \in \real{N}$ are defined by,
$$ X^N := (x_1, x_2, \ldots, x_N), \quad U^N := (u_0, u_1, \ldots, u_{N-1}) $$
\begin{remark}\label{rem:nouN}
There is no ``$u_N$'' in Problem ($P^N$) because of our use of a forward Euler method (Cf.~\eqref{eq:FE4xdot=f}).  The absence of $u_N$ is simply a characteristic of Radau-type discretizations.
\end{remark}

With foresight in mind, we discretize the adjoint equation by a backward Euler formula,
\begin{equation}\label{eq:adjoint-disc}
\begin{aligned}
\lambda_0 & = \lambda_1 + h \lambda_1 \partial_x f(x_1, u_1)\\
 & \vdots \\
\lambda_{N-2} &= \lambda_{N-1} + h \lambda_{N-1} \partial_x f(x_{N-1}, u_{N-1})
\end{aligned}
\end{equation}
In addition, we discretize \eqref{eq:HMC-FO} according to,
\begin{equation}\label{eq:gradH-disc}
\begin{aligned}
 \partial_u H(\lambda_0, x_0, u_0) & = 0 \\
 & \vdots \\
 \partial_u H(\lambda_{N-1}, x_{N-1}, u_{N-1}) & = 0
\end{aligned}
\end{equation}
Setting
\begin{equation}\label{eq:tvc-psi-disc}
\lambda_{N-1} = \partial_x E(x_N)
\end{equation}
a forward/backward Euler discretization of $(P^\lambda)$ may be written as,
\begin{eqnarray}
& (\textsf{$P^{\lambda N}$}) \left\{
\begin{array}{lrl}
& x_1 &= x^0 + h f(x^0, u_0) \\
& & \vdots \\
&x_N & = x_{N-1} + h f(x_{N-1}, u_{N-1})\\[1em]
&\lambda_0 & = \lambda_1 + h  \lambda_1 \partial_x f(x_1, u_1)\\
& & \vdots \\
&\lambda_{N-2} &= \lambda_{N-1} + h \lambda_{N-1} \partial_x f(x_{N-1}, u_{N-1})\\[1em]
&\lambda_{N-1} &= \partial_x E(x_N)\\[1em]
& 0 &= \partial_u H(\lambda_0, x_0, u_0) \\
& & \vdots \\
& 0 & = \partial_u H(\lambda_{N-1}, x_{N-1}, u_{N-1})
\end{array} \right. & \label{eq:probPpsiN}
\end{eqnarray}

\begin{remark}
The adjoint equation over the last interval $(t_N - t_{N-1})$ is missing in \eqref{eq:adjoint-disc} because its inclusion would require $u_N$, which does not exist; see Remark \ref{rem:nouN}.  Consequently, there is no $\lambda_N$ in ($P^{\lambda N}$).  Nonetheless, we still have an $N$-dimensional discretized dual variable given by $(\lambda_0, \lambda_1, \ldots, \lambda_{N-1}) $.
\end{remark}

In a conventional direct method, the objective is to solve Problem~$(P^N)$.  This optimization problem has $2N$ variables and $N$ constraints. A conventional indirect method involves solving the system of equations given by $(P^{\lambda N})$. This system has $3N$ variables and $3N$ constraints.  Different discretizations generate different versions of $(P^N)$ and $(P^{\lambda N})$.

\section{Development of the Main Proposition}
A direct shooting method attempts to solve a trajectory optimization problem using two main components: a simulation function and an optimization module\cite{betts-survey}.  The task of the simulation function is to eliminate the dynamics and provide the necessary data functions for the optimization module.  A basic direct shooting method in terms of solving $(P^N)$ can be described as follows:
\begin{enumerate}
\item For any given $U^N = (u_0, u_1, \ldots, u_{N-1})$, generate $x_N$ using the forward Euler method given by \eqref{eq:FE4xdot=f}.  This simulation function generates $x_N$ as a function of $U^N$.  By an abuse of notation, we denote this function as $x_N(U^N)$.
\item Starting with $i=0$, generate the next iterate $U^N_{i+1}$ according to the formula,
\begin{equation}\label{eq:family-disc}
U^N_{i+1} = U^N_i - \alpha_i\ M^{-1}\frac{\partial E^N}{\partial U^N_i}, \quad i = 0, 1, \ldots
\end{equation}
where $ i =0, 1, \ldots $ is the index of the iteration, $\alpha_i > 0$ is the step-length of the algorithm, $E^N$ is $E(x_N(U^N))$ and $M$ determines at least three algorithmic options given by the following choices\cite{GMSW,rossJCAM-1},
\begin{equation}\label{eq:3algs}
\begin{aligned}
M & = \text{ Identity matrix } & \text{(gradient method)}\\
M & = \text{ Hessian } & \text{(Newton's method)}\\
M & = \text{ Approximate Hessian } &\text{(quasi-Newton method)}
\end{aligned}
\end{equation}

\end{enumerate}
The dependence of $M$ on $U^N_i$ and $N$ is suppressed throughout this paper for enhancing the clarity of exposition of the main ideas.

Step 2 of the basic direct shooting method, encapsulated by \eqref{eq:family-disc}, is a more granular expression of the problem,
\begin{equation}
    (Q^{N}) \big\{  \mathop{\emph{Minimize }}_{U^N} E\big(x_N(U^N)\big) \label{eq:probQN}
\end{equation}
That is, Problem $(Q^N)$ is a transformation of $(P^N)$ generated as a result of Step~1, while \eqref{eq:family-disc} and \eqref{eq:3algs} are simply providing some pertinent details in solving \eqref{eq:probQN}.

Frequently, a direct shooting method is described more abstractly as follows: Parameterize the controls in terms of known functions $\xi_1(t), \xi_2(t), \ldots $ and unknown coefficients $c_1, c_2, \ldots$.  This can be accomplished in several different ways.  A representative parametrization is given by,
\begin{equation}\label{eq:u=sumcj}
u(t; C) = \sum_j c_j\xi_j(t)
\end{equation}
where, $C = (c_1, c_2, \ldots)$. Then, a more generic direct shooting method may be described as follows:
\begin{enumerate}
\item For any given $C$, generate $x(t_f)$ by simulating the differential equation $\dot x = f(x, u)$.  In terms of Step~1 of the basic direct shooting method described earlier, this process generates ``$x_N$'' as a function of $C$. Denote this function as $x_N(U^N(C))$.
\item With the dynamic constraints in Problem $(P^N)$ eliminated by the simulation function of Step~1, solve the resulting optimization problem given by,
\begin{equation}
    (Q^{N}_C) \big\{  \mathop{\emph{Minimize }}_{C} E\big(x_N(U^N(C))\big) \label{eq:probQNC}
\end{equation}
\end{enumerate}
If \eqref{eq:probQNC} is implemented in a manner analogous to \eqref{eq:family-disc}, it generates the family of algorithms parameterized by $M$ according to,
\begin{equation}\label{eq:family-disc-r}
C_{i+1} = C_i - \alpha_i\ M^{-1}\frac{\partial E^N}{\partial C_i}, \quad i = 0, 1, \ldots
\end{equation}

It is evident that the computation of the gradient of $E^N$ is central to all variants of the direct shooting methods discussed in the preceding paragraphs.  More specifically, \eqref{eq:family-disc} requires the computation of the gradient of $E^N$ with respect to $U^N$ while \eqref{eq:family-disc-r} requires this same information albeit implicitly because
\begin{equation}\label{eq:gradEwrtC}
\frac{\partial E^N}{\partial C_i} = \left[\frac{\partial U^N_i}{\partial C_i}\right]^T \frac{\partial E^N}{\partial U^N_i}
\end{equation}
%
\begin{remark}\label{rem:EN-notation}
To avoid an excessive use of new symbols, we have used the notation $E^N$ in \eqref{eq:family-disc} to imply the composite function $E \circ x_N: U^N \mapsto \Real$ in \eqref{eq:probQN}.  By the same token, $E^N$ in \eqref{eq:family-disc-r} is a different composite function given by $E \circ x_N \circ U^N: C \mapsto \Real$.
\end{remark}
\begin{remark}
The versatile and popular Program to Simulate Optimize Simulated Trajectories II (POST2)\cite{POST2ref} uses a direct shooting method with controls parameterized by various known functions and unknown coefficients analogous to \eqref{eq:u=sumcj}.  Step~2 in POST2 (i.e., the equivalent of solving Problem $(Q^N_C)$) is implemented in terms of a gradient or a quasi-Newton method\cite{POST2ref}.
\end{remark}

\begin{lemma} Let $E^N$ denote the function $E\circ x_N: U^N \mapsto \Real$.  Then, for $k = 0, 1, \ldots (N-1)$,
\begin{equation}
\frac{\partial E^N}{\partial u_k} = h \lambda_k \partial_u f(x_k, u_k)
\end{equation}
\end{lemma}
\begin{proof}
By chain rule, we have,
\begin{align}
\frac{\partial E^N}{\partial u_k}  &= \frac{\partial E}{\partial x_N} \frac{\partial x_N}{\partial u_k} \nonumber \\
    & = \lambda_{N-1} \frac{\partial x_N}{\partial u_k} \label{eq:1inLemma1}
\end{align}
where the last equality in \eqref{eq:1inLemma1} follows from \eqref{eq:tvc-psi-disc}. The remainder of the proof is broken down to two cases:

\underline{Case $k=N-1$}

From \eqref{eq:FE4xdot=f} it follows that,
\begin{equation}\label{eq:lastGradproof}
\frac{\partial x_N}{\partial u_{N-1}} =  h \partial_u f(x_{N-1}, u_{N-1})
\end{equation}
Thus, from \eqref{eq:1inLemma1} and \eqref{eq:lastGradproof}, it follows that the statement of the lemma holds for $k = N-1$.

\underline{Case $k=N-2$}

From \eqref{eq:FE4xdot=f} we have
\begin{align}
\frac{\partial x_N}{\partial u_{N-2}} &=  \frac{\partial x_{N-1}}{\partial u_{N-2}} +  h \partial_x f(x_{N-1}, u_{N-1}) \frac{\partial x_{N-1}}{\partial u_{N-2}} \nonumber \\
&=  \frac{\partial x_{N-1}}{\partial u_{N-2}}\Big( 1 +  h \partial_x f(x_{N-1}, u_{N-1}) \Big) \nonumber \\
& =h \partial_u f(x_{N-2}, u_{N-2}) \Big( 1 +  h \partial_x f(x_{N-1}, u_{N-1}) \Big) \label{eq:sad-1}
\end{align}
Substituting \eqref{eq:sad-1} in \eqref{eq:1inLemma1} we get,
\begin{align}
\frac{\partial E^N}{\partial u_{N-2}} &= h \partial_u f(x_{N-2}, u_{N-2})\Big( \lambda_{N-1} +  h \lambda_{N-1}\partial_x f(x_{N-1}, u_{N-1}) \Big) \nonumber \\
&=  h \lambda_{N-2} \partial_u f(x_{N-2}, u_{N-2}) \label{eq:happy-1}
\end{align}
where the last equality in \eqref{eq:happy-1} follows from \eqref{eq:adjoint-disc}.

The cases for $k = (N-3), (N-2), \ldots, 0$ all follow similarly with additional steps in chain rule and composition.
\end{proof}
\begin{remark}
Lemma 1 requires that the adjoint equation be discretized by the backward Euler formula; see the statement preceding \eqref{eq:adjoint-disc}.
\end{remark}
\begin{lemma}  The gradient of the cost function is equal to the gradient of the Hamiltonian scaled by the integration step size $h$:
\begin{equation}\label{eq:gradE=hgradH}
\frac{\partial E^N}{\partial U^N} = h \left[
                                  \begin{array}{c}
                                    \partial_u H(\lambda_0, x_0, u_0)  \\
                                    \vdots \\
                                     \partial_u H(\lambda_{N-2}, x_{N-2} u_{N-2}) \\
                                     \partial_u H(\lambda_{N-1}, x_{N-1}, u_{N-1}) \\
                                  \end{array}
                                \right]
\end{equation}
\end{lemma}
\begin{proof}
This result follow readily from Lemma 1 and \eqref{eq:H:=}.
\end{proof}
%
\begin{proposition}\label{prop:shoot=equiv}
A direct shooting method is first-order equivalent to an indirect method.
\end{proposition}
\begin{proof}
Consider Step 2 of a direct shooting method.  In particular, consider the computation of the gradient of $E^N$.  From Lemmas 1 and 2, it follows that the computation of the last row of the left hand side of  \eqref{eq:gradE=hgradH} is mathematically equivalent to the following steps:
\begin{enumerate}
\item Compute the gradient of the cost function $\partial_x E(x_N)$.  Set it equal to  $\lambda_{N-1}$.  This equality follows from the transversality condition given by \eqref{eq:tvc-psi-disc}.
\item Using the result from Step 1, compute the gradient of the Hamiltonian at the grid point $t_{N-1}$ according to the formula,
$$  \partial_u H(\lambda_{N-1}, x_{N-1}, u_{N-1}) \leftarrow \lambda_{N-1} \partial_u f(x_{N-1}, u_{N-1})$$
\item Multiply the result from Step 2 by the integration step size $h$.
\end{enumerate}
By similar arguments, it follows that computation of the remainder of the rows of the left hand side of \eqref{eq:gradE=hgradH} is equivalent to the following steps:
\begin{enumerate}
\item Evaluate the gradient of the cost function $\partial_x E(x_N)$ and set it equal to  $\lambda_{N-1}$ (Cf.~\eqref{eq:tvc-psi-disc}).
\item Using the result from Step~1, back-propagate the costate using the adjoint equation (Cf.~\eqref{eq:adjoint-disc}) to determine $\lambda_{N-2}, \ldots, \lambda_0$.
\item Using \eqref{eq:H:=}, evaluate the gradient of the Hamiltonian at grid points $t_{N-2}, \ldots, t_0$.
\item Multiply the result from Step 3 by $h$ to compute the gradient of the cost function.
\end{enumerate}
The preceding steps are, in fact, the classical steps of an indirect  method.
\end{proof}
\begin{remark}
Equation \eqref{eq:gradE=hgradH} is a mathematically factual statement.  It does not necessarily imply a formula to generate the left-hand side of \eqref{eq:gradE=hgradH} via its right-hand side as used in the proof of Proposition~1. The left-hand side of \eqref{eq:gradE=hgradH} can be generated independently; for instance, via numerical finite differencing as in the case of POST2\cite{POST2ref}.  Regardless of how the left-hand side of \eqref{eq:gradE=hgradH} is produced, Lemma~2 and hence Proposition~1 holds.
\end{remark}

\section{New Results on Convergence and Accuracy}
Proposition \ref{prop:shoot=equiv} casts serious doubts on a number of historical and anecdotal claims on the convergence and accuracy of direct and indirect methods. In order to place these claims in a proper context, we define the following notions of algorithms and convergence.
\begin{definition}[Inner Algorithm and Its Convergence]
Let $\mathcal{A}$ be an algorithm for solving Problem $(P^N)$.  That is, for a fixed $N$, $\mathcal{A}$ generates a sequence of vector pairs $(X^N_0, U^N_0), (X^N_1, U^N_1), \ldots $ by the iterative map $(X^N_{i+1}, U^N_{i+1}) = \mathcal{A}(X^N_i, U^N_i), \ i \in \mathbb{N}$.  Algorithm $\mathcal{A}$ is said to converge if
$$\lim_{n \to \infty} X^N_n := X^N_\infty , \qquad \lim_{n \to \infty} U^N_n  := U^N_\infty $$
is an accumulation point that solves Problem $(P^N)$. Furthermore, $\mathcal{A}$ is called an inner algorithm for Problem~$(P)$.
\end{definition}
\begin{definition}[Convergence of a Discretization]
Let $(X^N_\infty, U^N_\infty)$ be a solution to Problem~$(P^N)$.  A discretization is said to converge if
$$\lim_{N \to \infty} X^N_\infty := X^\infty_\infty , \qquad \lim_{N \to \infty} U^N_\infty := U^\infty_\infty $$
is an accumulation point that solves Problem $(P)$.
\end{definition}
\begin{definition}[Trajectory Optimization Algorithm]
Let $\mathcal{B}$ be an algorithm that generates a sequence of integer pairs $(N_0, m_0), (N_1, m_1) \cdots $ such that the sequence
$$(X^{N_0}_{m_0}, U^{N_0}_{m_0}), (X^{N_1}_{m_1}, U^{N_1}_{m_1}), \ldots  $$
converges to an accumulation point $(X^\infty_\infty, U^\infty_\infty)$ given by Definition~2. Then, $\mathcal{B}$ is called a trajectory optimization algorithm for solving Problem $(P)$.
\end{definition}
Note that the preceding definitions are agnostic to a direct shooting method.  Also, for the purposes of brevity, we have taken the liberty of avoiding the specifics of a metric for convergence in Definitions 1, 2 and 3.
\begin{remark}
Algorithm $\mathcal{B}$ may be designed using $\mathcal{A}$. In this design, it may be advantageous to have $\mathcal{A}$ depend upon $N$; i.e., different algorithms for different choices of $N$.  For example, the inner algorithm for $N = N_0$ may need to be robust rather than fast while for $N \ne N_0$, it may be advantageous to design $\mathcal{A}$ that is fast rather than robust.  Furthermore, it is possible to design $\mathcal{B}$ with $\mathcal{A}$ not even executed through its completion as implied in Defintion~3.  Although it is not a direct method, the spectral algorithm\cite{spec-alg,PSReview-ARC-2012,arb-grid} implemented in DIDO\cite{dido} incorporates such a variable inner algorithm to support its guess-free execution\cite{guess-free} via an algorithm $\mathcal{B}$ as given by Definition~3.
\end{remark}

\subsection{Theoretical Analysis}
Consider the convergence of a shooting method that uses a gradient-based algorithm. The convergence of such algorithms can be better understood by considering \eqref{eq:family-disc} to be a forward Euler discretization of the ``continuous-time'' dynamical equation given by,
\begin{equation}\label{eq:grad-cont}
\frac{dU^N}{d\tau} = - M^{-1} \frac{\partial E^N}{\partial U^N}
\end{equation}
That is, it is apparent that discretizing \eqref{eq:grad-cont} by a forward Euler method using a step size $\alpha_i$ generates \eqref{eq:family-disc}.  Consequently, the family of algorithms given by \eqref{eq:family-disc} is equivalent to a simple forward propagation of \eqref{eq:grad-cont}; i.e., an initial value problem.  Thus, the convergence of \eqref{eq:family-disc} can be separated into two clear subproblems:
\begin{enumerate}
\item Stability of the differential equation given by \eqref{eq:grad-cont}, and
\item Stability of the forward-Euler discretization of \eqref{eq:grad-cont}.
\end{enumerate}
If the dynamical system given by \eqref{eq:grad-cont} is unstable, then its corresponding algorithm will diverge no matter the step size. Thus, a necessary condition for the convergence of an inner algorithm (see Definition~1) is that \eqref{eq:grad-cont} be stable.  Consequently, the full power of dynamical system theory can be brought to bear in analyzing the stability of \eqref{eq:grad-cont}.
\begin{remark}
\textbf{Analyzing the stability of \eqref{eq:grad-cont} is an open area of research}.  This is because $E^N$ is a composite function (see Remark \ref{rem:EN-notation}) that depends on the data functions $E$ and $f$ as well as the type of discretization used in generating $E^N$. Furthermore, if $E^N$ is generated using an adaptive grid, then it is also a function of the grid itself.
\end{remark}
%
Now suppose that \eqref{eq:grad-cont} is stable; then, the  ``continuous trajectory of the algorithm'' $\tau \mapsto U^N$ converges to some equilibrium point $U^N_\infty$. By definition, this equilibrium point satisfies the condition $dU^N/d\tau =~0$; hence, by \eqref{eq:grad-cont}, if $M^{-1}$ is nonsingular, $U^N_\infty$ satisfies the algebraic equation,
\begin{equation}\label{eq:Fermat4QN}
\left.\frac{\partial E^N}{\partial U^N}\right|_{U^N_\infty} =0
\end{equation}
Equation \eqref{eq:Fermat4QN} is precisely a first-order necessary condition for Problem~$(Q^N)$ which is the optimization problem in Step~2 of the basic direct shooting method; see \eqref{eq:probQN}.
\begin{remark}
If Step 2 of a direct shooting method is implemented to solve Problem $(Q^N_C)$ defined in \eqref{eq:probQNC}, then the equation corresponding to \eqref{eq:Fermat4QN} is given by
\begin{equation}\label{eq:Fermat4QNC}
\left.\frac{\partial E^N}{\partial C}\right|_{C_\infty} =0
\end{equation}
Equation \eqref{eq:Fermat4QNC} does not necessarily imply \eqref{eq:Fermat4QN} unless $\partial_C U^N$ has full rank; see \eqref{eq:gradEwrtC}. In other words, to avoid spurious solutions, the control must be parameterized such that $\partial_{c_i} u(t) $ must generate a full rank condition for $\partial_C U^N$.
\end{remark}

The preceding discussion illustrates how the convergence of algorithms can be analyzed by using well-established theorems related to the stability of differential equations.  The use of differential equations to analyze convergence of algorithms is not entirely new; it goes as far back as 1958 to the pioneering work of Gavurin\cite{gavurin}.   Recently, these ideas have reemerged as a new field for designing and analyzing optimization algorithms\cite{rossJCAM-1,su,polyak-ode,shi:hi-res-Nes,ross-accel}.

Suppose that the inner gradient-based algorithm converges to the equilibrium point $U^N_\infty$; then, by Lemma 2 and \eqref{eq:Fermat4QN} we have,
\begin{equation}\label{eq:Hu_k=0}
\partial_u H(\lambda_k, x_k, u_k) = 0 \quad \text{at } U^N_\infty \text{ for } k = 0, 1, \ldots, (N-1)
\end{equation}
Thus, if the direct shooting method converges (in the sense of Definition 1), it converges to a point that satisfies the discretized ``indirect'' first-order necessary conditions given by $(P^{\lambda N})$; i.e., the same point as an indirect method.  If the discretization converges (see Definition~2), then a solution from the direct shooting method will converge to a solution of $(P^\lambda)$.  \textbf{\emph{In this theoretical context, the accuracy of a direct shooting method is the same as an indirect method}}.

\subsection{Practical Limits}
Despite the favorable theoretical results of the preceding paragraphs, practical experiences can be dramatically different based on their implementations.  This is because practical algorithms satisfy convergence conditions only up to some small tolerances; hence, the theoretical results must be reanalyzed in this context.  To this end suppose that \eqref{eq:Fermat4QN} is satisfied to within some $\epsilon > 0$; i.e.,
\begin{equation}\label{eq:Fermat=eps}
0 < \left|\frac{\partial E^N}{\partial U^N}\right|_\infty \le \epsilon
\end{equation}
From Lemma 2, it follows that \eqref{eq:Fermat=eps} maps to
\begin{equation}\label{eq:Hu_k=epsbyh}
\left|\partial_u H(\lambda_k, x_k, u_k)\right| \le \epsilon/h \quad \text{ for } k = 0, 1, \ldots, (N-1)
\end{equation}
Equations \eqref{eq:Fermat=eps} and \eqref{eq:Hu_k=epsbyh} generates the following theorem that summarizes the conditions and interplay between convergence and accuracy:
\begin{theorem}\label{theorem:1}
Suppose that the first-order optimality condition given by \eqref{eq:Fermat4QN} is satisfied to a tolerance of $\epsilon > 0$ (Cf.~\eqref{eq:Fermat=eps}).  Then, the first-order Hamiltonian optimality condition (Cf.~\eqref{eq:gradH-disc}) is satisfied to a tolerance of $\epsilon/h$, where $h > 0$ is the integration step size.
\end{theorem}
\begin{proof}
The proof of this theorem follows as a direct consequence of \eqref{eq:Fermat=eps} and \eqref{eq:Hu_k=epsbyh}.
\end{proof}

\section{Resetting Some Facts and Myths About Direct and Indirect Methods}\label{sec:reset}
Historically, direct and indirect methods have been viewed as separate computational methods with pros and cons in terms of convergence, accuracy, required human-labor etc.\cite{vonStryk-survey,betts-survey}. Proposition~\ref{prop:shoot=equiv} and Theorem~\ref{theorem:1} suggest otherwise.  Interestingly, Proposition~1 and Theorem~1 can also help explain the past computational discrepancies between direct and indirect methods. These points are  summarized in the following subsections.

\subsection{Accuracy}
Classical direct methods have been considered inaccurate with respect to an indirect method. This ``fact'' can be easily explained via Theorem~1.   For example, if a step size of $h = 0.01$ is used for integration, then the accuracy of the first-order Hamiltonian minimization condition is given by $\epsilon/h = 100 \epsilon$, or one-hundred times worse than the tolerance on the gradient of the cost function! In fact, as the integration accuracy is increased (i.e., as $h \to 0$ for convergence in the sense of Definition~2), the practical satisfaction of the first-order Hamiltonian optimality condition given by \eqref{eq:Hu_k=epsbyh} worsens because $\epsilon/h \to~\infty$ for a fixed value of $\epsilon$.  This explains why a direct shooting method has been considered to be inaccurate with respect to optimality. One simple remedy for this problem is to coordinate the practical convergence of the algorithm with that of the discretization in accordance with Theorem~1; i.e., vary $\epsilon$ in concert with $h$ so that $\epsilon/h$ is as small as desired.

\subsection{Rate of Convergence}
Indirect methods are considered to be superior to direct methods in terms of their rates of convergence.  This perception can be explained as follows: From Lemma 2 and \eqref{eq:family-disc}, the inner gradient-based algorithm corresponding to a direct shooting method can be written as,
\begin{equation}\label{eq:step-size-reduce}
U^N_{i+1} = U^N_i - \alpha_i\ h\ M^{-1}\left(\frac{\partial H}{\partial U^N_i}\right), \quad i = 0, 1, \ldots
\end{equation}
Hence, if a classical direct method is implemented with a step size of say, $\alpha_i = 1$, it follows from \eqref{eq:step-size-reduce} that this is equivalent to implementing an indirect method with a step size of $h$.  This fact explains the observation that as the integration step size is reduced to enhance accuracy, the observed convergence of a classical direct method drops quite dramatically (assuming $h < 1$).
An obvious remedy for this problem is adjust the step size of  the gradient algorithm such that $\alpha_i h =1 $.  Alternatively, one could scale the gradient of the direct shooting method by $1/h$; however, we caution that scaling an optimal control problem must be done in a careful manner that ensures that the dual variables do not get imbalanced.  See \cite{scaling} for further details.

\subsection{``Radius''/Region of Convergence}
A common argument against an indirect method is that it has a small ``radius'' of convergence relative to a direct method. This argument seems to be at loggerheads with Proposition~1. Nonetheless, Theorem~1 together with the emerging new theory of optimization\cite{rossJCAM-1} provides the answer. As noted in the discussions following \eqref{eq:grad-cont}, an optimization algorithm may be viewed as a discretization of a differential equation.  A necessary condition for the algorithm to converge is that the Euler step size of its corresponding differential equation be stable.  That is, a smaller step size promotes global convergence, while a larger step size contributes to instability. This is precisely what occurs in the direct shooting method when compared to its indirect counterpart. From \eqref{eq:step-size-reduce}, it is clear that the effective step size of the direct shooting method is less than $\alpha_i$ if $h < 1$; hence, it promotes global convergence at the price of an observed decreased rate of convergence.  A simple remedy to increase the size of the basin of attraction of an indirect method is to reduce the algorithm step size to bring it on par with that of its direct counterpart.

\subsection{Human Labor}
A common dissatisfaction associated with an indirect method is the reported human-labor in arriving at the necessary conditions. A quick examination of $(P^{\lambda N})$ shows that the development of the necessary conditions is equivalent to the computation of the gradients/Jacobians. These computations can be automated to the point of completely eliminating human labor\cite{ross-book,trelat-survey}.  In fact, the gradients/Jacobians need not even be computed symbolically/analytically from the point of view of human-labor intensity.  Furthermore, when a software that purports to be a direct method requires Jacobian information, it is effectively asking a user to provide the first-order necessary conditions. In other words, \textbf{\emph{the historical dissatisfaction of an indirect method in terms of human-labor is more a reflection of an architecture of the software (and its sophistication or lack thereof) rather than the fundamentals of the trajectory optimization method}}.

\subsection{Challenges and Failures of Adaptive Grids}
All of the results presented so far apply to a uniform grid.  Extending the analysis to an arbitrary grid requires new mathematical results that are beyond the scope of this paper.  It is also an open area of research. Nonetheless, it is possible to show rather quickly why an adaptive grid causes failures in a direct shooting method. To this end, let $\pi^N= [t_0, t_1, \ldots, t_N]$ be a nonuniform grid.  Then, $x_N$ not only depends upon $U^N$ but also on $\pi^N$; hence, under the caveat noted in Remark~\ref{rem:EN-notation}, we set
\begin{equation}\label{eq:EN-var-grid}
E^N \equiv E(x_N(U^N, \pi^N))
\end{equation}
Under the same notational caveat, we set,
\begin{equation}\label{eq:piN-adapt}
\pi^N \equiv \pi^N(U^N)
\end{equation}
to denote an adaptive grid; i.e., the grid $\pi^N$ is changed over the course of one or more iteration cycles.  Using \eqref{eq:EN-var-grid} and \eqref{eq:piN-adapt}, the gradient of the endpoint cost function may be written as,
\begin{align}
\frac{\partial E^N}{\partial U^N} &= \frac{\partial E^N}{\partial x_N} \left(\frac{\partial x_N}{\partial U^N}   + \left[\frac{\partial \pi^N}{\partial U^N}\right]^T \frac{\partial x_N}{\partial \pi^N}  \right) \nonumber\\
&= \frac{\partial E^N}{\partial x_N} \frac{\partial x_N}{\partial U^N}   + \underbrace{\frac{\partial E^N}{\partial x_N} \left[\frac{\partial \pi^N}{\partial U^N}\right]^T \frac{\partial x_N}{\partial \pi^N}}_S  \label{eq:gradEN+noise}
\end{align}
where, $S$ symbolizes the second term on the right-hand-side of \eqref{eq:gradEN+noise}. This term is zero in Lemma~1 and Theorem~1. Taking $S$ into account requires specific details of adaptation. Even under an assumption of these details, the computation of Jacobian of $\pi^N$ may not be an easy task.

If $S$ is not taken into account in a direct shooting method that uses a variable step size, then an inner algorithm based solely on the first term of the right-hand-side of \eqref{eq:gradEN+noise} will have the wrong gradient information. This is precisely the source of the problem in shooting methods based on variable step sizes.  In fact, practitioners have long cautioned on the use of variable step size integrators because of ``noise'' in the gradients\cite{betts-survey}.  Equation \eqref{eq:gradEN+noise} provides a more precise statement of ``noise'' in terms of the second term on the right-hand-side of \eqref{eq:gradEN+noise}.

\begin{remark}
Adaptive grids also create new issues such as the introduction of unnecessary dynamics and unaccounted feedback loops that may destabilize the resulting algorithm for solving even a simple problem.  See \cite{scaling} for details.
\end{remark}
%
From \eqref{eq:gradEN+noise} and the analysis presented in \cite{scaling}, it is evident that a better choice for an adaptive grid is to vary it with respect to $N$ as part of algorithm $\mathcal{B}$ while decoupling its variation with respect to the inner algorithm $\mathcal{A}$ (see Definitions 1 and 3).  This point was also observed in \cite{kt-map} as part of the outcomes of numerical experimentations with adaptive pseudospectral grids; however, its cause at that time was unknown.  Evidently, a more rigorous reexamination of several entrenched concepts in trajectory optimization is warranted.

\section{The Case For Hamiltonian Programming}

Based on the analysis presented in Section~\ref{sec:reset}, it is clear that if Theorem~\ref{theorem:1} is violated, then a direct shooting method exhibits all the properties reported in the literature.  A simple remedy to avoid the reported problems is to not violate Theorem~\ref{theorem:1}. There are several ways to comply with Theorem~\ref{theorem:1} some of which are outlined in Section~\ref{sec:reset}.  Any method that complies with Theorem~\ref{theorem:1} incorporates its indirect counterpart; hence, we call such methods Hamiltonian programming (even if a Hamiltonian is not used explicitly).

\subsection{Introduction to Hamiltonian Programming}

We use the term Hamiltonian programming in a much broader sense than complying with the specifics of Theorem~\ref{theorem:1}. To elaborate this point, note that Proposition~\ref{prop:shoot=equiv} is essentially a statement of a covector mapping theorem for a shooting method. That is, Proposition~\ref{prop:shoot=equiv} may be construed as an application of the covector mapping principle (CMP)\cite{ross-book} which basically states that there exists a connection between a direct and an indirect method.  As noted in Ross~et~al.~\cite{scaling}, this connection is a direct consequence of the well-known \emph{Hahn-Banach theorem}\cite{Tao:epsilon}.  To illustrate an application of the Hahn-Banach theorem and the resulting CMP more concretely, consider a linear ordinary differential equation (ODE),
\begin{equation}\label{eq:linear-ode}
\dot\bx = \bA \bx
\end{equation}
We can always associate with \eqref{eq:linear-ode} a ``shadow'' ODE given by,
\begin{equation}\label{eq:adjoint-ode}
-\dot\bpsi = \bA^T \bpsi
\end{equation}
That is, \eqref{eq:adjoint-ode} exists from the mere fact that \eqref{eq:linear-ode} exists.  If \eqref{eq:linear-ode} is nonlinear, then \eqref{eq:adjoint-ode} still exists with $\bA$ replaced by the Jacobian of the right-hand-side of the nonlinear ODE.  This follows from the following ``Hahn-Banach-Hamiltonian algorithm:''
\begin{enumerate}
\item Given an ODE, $\dot\bx = \bff(\bx)$, construct a scalar function $\mathcal{H}$ according to the formula,
\begin{equation}\label{eq:Hamil=}
\mathcal{H}(\bpsi, \bx) := \bpsi^T\bff(\bx)
\end{equation}
where $\bpsi \in \real{N_x}$ is some auxiliary vector that has the same dimension as $\bx \in \real{N_x}$.
\item Construct an auxiliary differential equation using $\mathcal{H}$ from Step~1 by performing the following operation,
\begin{equation}
 -\dot\bpsi := \partial_{\bx} \mathcal{H}(\bpsi, \bx) = \left(\frac{\partial \bff}{\partial\bx}\right)^T \bpsi
\end{equation}
\end{enumerate}
From \eqref{eq:Hamil=}, it follows that the ODE $\dot\bx =\bff(\bx)$ can also be written as $\dot\bx =  \partial_{\psi} \mathcal{H}(\bpsi, \bx)$; hence, any differential equation may be viewed as ``half'' of a Hamiltonian system given by:
\begin{subequations}\label{eq:ode=Hamil}
\begin{align}
\dot\bx &= \partial_{\psi} \mathcal{H}(\bpsi, \bx) &(\Rightarrow \dot\bx = \bff(\bx))\\
-\dot\bpsi &= \partial_{\bx} \mathcal{H}(\bpsi, \bx)&(\text{shadow ODE}) \label{eq:shadow-ode}
\end{align}
\end{subequations}
Alternatively, \textbf{\emph{any differential equation $\dot\bx = \bff(\bx)$ may be Hamiltonianized according to \eqref{eq:ode=Hamil}}}. In this context, Proposition~\ref{prop:shoot=equiv} should not be a surprise in the sense that it is simply providing the ``missing'' details of the other half of the Hamiltonian system associated with a direct shooting method.  This is how a CMP equates a direct method with an equivalent indirect method\cite{ross-book}. Once such a constructive equivalence is established, an algorithm for a direct method can be modified (i.e., \emph{Hamiltonianized}) so that it converges to a solution of its indirect counterpart. \textbf{\emph{Consequently, it is unnecessary, or even improper, to classify trajectory optimization methods as direct or indirect}}.  See also Sec.~VI.D.

\subsection{Hamiltonian vs Nonlinear Programming}
In constructing numerical methods for solving ODEs, one can, in principle, ignore its dual counterpart (i.e., the shadow ODE given by \eqref{eq:shadow-ode}). \emph{\textbf{Ignoring the shadow ODE in computational optimal control is tantamount to ignoring 50\% of the differential equations that define the system! }}In this context, it should not at all be surprising that if a method is constructed that takes into account 100\% of the differential equations it will most certainly perform better than one that accounts for only half the equations. One can go even further and state that if a method is constructed that accounts for only 50\% of its ODEs, then it may even be incorrect.  \emph{A Hamiltonian programming method accounts for 100\% of the ODEs}.

A classical indirect method is obviously Hamiltonian: it ``directly'' accounts for 100\% of the governing ODEs.  A classical direct method only accounts for 50\% of the ODEs. Nonetheless, a direct method can be Hamiltonianized by (``indirectly'') incorporating the missing 50\% of the ODEs (Cf.~Theorem~\ref{theorem:1}). In incorporating these ideas with those presented in Sec.~VI.D, it follows that trajectory optimization is fundamentally Hamiltonian while nonlinear programming is Lagrangian. Thus, the many issues with classical direct methods reported in the literature (see Sec.~\ref{sec:reset}) can now be framed simply as a result of non-Hamiltonian programming.

Under this backdrop, we declare Problem~$(P^N)$ to be a Hamiltonian programming problem rather than a nonlinear programming problem. Such a declaration at the starting point immediately avoids the many pitfalls associated with the preconceptions of nonlinear programming.  In this context, we note the following:
\begin{enumerate}
\item \textbf{\emph{A Hamiltonian is not a Lagrangian}}. Nonlinear programming methods are based on Lagrangians. Hamiltonian programming methods are centered on Hamiltonians.
\item A nonlinear programming method is agnostic to the differences between (discretized) state and control variables. 
    A Hamiltonian programming method treats state and control variables differently and in accordance with their respective dependencies (mapping) to a Hamiltonian.
\item A Hamiltonian programming method treats a dynamic constraint differently than a path constraint. In sharp contrast, collocation-based nonlinear programming methods for optimal control treat all constraints the same way (except to exploit linearity/sparsity).
\item The concept of time and its ``hidden convexity''\cite{boris:hiddenC} is absent in a nonlinear programming method.  For instance, \textbf{\emph{the constancy of a Hamiltonian cannot be derived from the necessary conditions for Problem~$(P^N)$ }} if it treated as a nonlinear programming problem unless each point on the grid $\pi^N$ is considered an optimization variable in an open subinterval.
\end{enumerate}
Defining Problem~$(P^N)$ to be a Hamiltonian programming problem is in line with labeling special mathematical programming problems as subjects deserving their own monikers, theory and algorithms.  Examples are integer programming or linear programming.  Furthermore, even special nonlinear programming problems (e.g. cone and quadratic programming problems) are treated very differently than generic ones.
%
Thus, for example, if a convex optimization problem is solved using generic nonlinear programming methods, it can certainly be successful albeit in a limited sense.  Because the limitations of such an approach does not imply convex programming problems are hard, the same is true when generic nonlinear programming methods are used to solve Problem~$(P^N)$.   Consequently, the more suitable method for solving Problem~$(P^N)$ is to use Theorem~\ref{theorem:1} to expose its hidden Hamiltonian structure and subsequently design \emph{\textbf{Hamiltonian algorithms}} that exploit the Hamiltonian conditions.  As evident from the proofs of Proposition~\ref{prop:shoot=equiv} and Theorem~\ref{theorem:1}, \emph{\textbf{these aspects of mathematical programming are well beyond the apparent sparsity or linearity in viewing $(P^N)$ as a special nonlinear programming problem}}.

Discretizing a generic optimal control problem naturally generates a  matrix-vector Hamiltonian programming problem rather than a ``sparse'' nonlinear programming problem\cite{birk-TN}. Furthermore, the rows of the matrix in this formulation have a separable-programming property\cite{GMSW} that can be exploited for fast and efficient computation\cite{dido}.

\subsection{Growing Literature on Hamiltonian Programming}

A recognition of the fact that discretizing optimal control problems may generate mathematical programming problems that are outside the reach of nonlinear programming theory is not new.  It goes back to the early days of pseudospectral\cite{lncis,acc:hybrid} and ``orthogonal'' collocation methods\cite{biegler2002}.   Collocation methods for optimal control\cite{biegler2002,DynoPC} are similar to pseudospectral knotting methods\cite{knots,auto-knots} although these two methods were developed independently with applications to problems in chemical engineering and aerospace guidance respectively. In the methods advanced by Biegler et al\cite{biegler2002,biegler2014,biegler2016} and Ross et al\cite{arb-grid,dido,acc:hybrid,lncis}, certain Hamiltonian-type conditions are incorporated in generating candidate solutions. Emerging new ideas incorporate additional Hamiltonian conditions in several different ways whose discussion is beyond the scope of this paper; see \cite{ross-book,birk-TN,arb-grid, scaling,biegler2014, biegler2016,hamilMeshRef, marshDAE} for details.  An instantiation of some of these ideas are implemented in the software packages DynoPC\cite{DynoPC} and DIDO\cite{dido} which can reportedly solve ``hard'' optimal control problems that cannot be solved by non-Hamiltonian programming methods\cite{biegler2014,biegler2016,marshDAE}.  In particular, DIDO interfaces with a user similar to the technicalities of a traditional direct method but also outputs all of the additional dual-space information (e.g. Hamiltonian) even though a user does not supply any necessary conditions.  This is because, as noted in Section~VI.D, \textbf{\emph{all of the human labor in the development of necessary conditions can indeed be eliminated by automation}}.  Furthermore, because DIDO implements a sequential Hamiltonian programming algorithm\cite{dido} in the sense of Definition~3, it is robust (in a new Lyapunov sense\cite{rossJCAM-1}) relative to variations in the staring point of the algorithm; i.e., a guess.  Elastic programming\cite{spec-alg,guess-free} furthers the robustness property of DIDO to the extent that it runs without a user-supplied guess\cite{guess-free} while continuing to automatically output necessary Hamiltonian conditions similar to an indirect method.  In other words, \textbf{\emph{DIDO is a ``proof of existence'' of the statement that ``direct = indirect'' methods for trajectory optimization}} provided that the generic nonlinear programming method is replaced or augmented by Hamiltonian programming techniques.

\section{A Brief Historical Perspective on The Direct-Indirect Chasm}
Although Section VI explains the mathematical origins of the direct-indirect divide through a narrow prism of direct shooting methods, it is instructive to view the broader chasm from a historical perspective. To set the stage for a brief historical review, we recall from Section V that an optimization method may be considered as a discretization of a ``continuous-time'' dynamical system whose convergence may be analyzed in terms of Lyapunov stability\cite{rossJCAM-1,polyak-ode}. A trajectory optimization problem generates functions for an intermediate optimization problem (see Definition~1). If a trajectory optimization method that produces these functions generates an unstable dynamical system for optimization, then the resulting inner algorithm will diverge.  For instance, for the basic direct shooting method, properties of the generated function $E^N: U^N \mapsto \Real$ that define Problem $(Q^N)$ (see \eqref{eq:probQN}) drive the attributes of the inner optimization algorithm.  If the continuous-time dynamical system for optimization (Cf.~\eqref{eq:grad-cont}) is stable, then the resulting algorithm may still diverge if the step size of the ``discrete'' algorithm is too large.   \textbf{\emph{An analysis of trajectory optimization methods using these new ideas is an open area of research.}}

Historically, trajectory optimization methods were generated as direct and direct methods based on continuous-time considerations\cite{vonStryk-survey,betts-survey} and not from the perspective of the arguments presented in this paper. For example, it is straightforward to show\cite{brysonHo,kirk} that the first variation $\delta J$ for Problem~$(P)$ can be written as,
\begin{equation}\label{eq:first-var}
\delta J = \int_{t_0}^{t_f} \left(\frac{\partial H}{\partial u} \right)\delta u \, dt
\end{equation}
where $\delta u$ is the variation of $u(\cdot)$ such that all other constraints are satisfied. Hence, if $\delta u$ is selected as,
\begin{equation}\label{eq:fromKirk}
\delta u = - \gamma \left(\frac{\partial H}{\partial u} \right), \quad \gamma > 0
\end{equation}
it follows that \eqref{eq:first-var} reduces to
\begin{equation}
\delta J = -\gamma \int_{t_0}^{t_f} \left(\frac{\partial H}{\partial u} \right)^2 \, dt \quad \Rightarrow\quad  \delta J \le 0
\end{equation}
Thus, $\delta J$ decreases with the selection of \eqref{eq:fromKirk} and vanishes if $\partial_u H$ vanishes.  Consequently, \eqref{eq:fromKirk} generates the variational equivalent of a gradient method which can be discretized to generate a computational method according to,
\begin{equation}
\delta u _k = - \gamma   \left(\frac{\partial H}{\partial u} \right)_k, \quad k \in \mathbb{N}
\end{equation}
Likewise a second-order method may be obtained by pre-multiplying the gradient by an appropriate Hessian\cite{brysonHo}.  Although these ideas are well-founded, they mask the higher-level of granularity offered by \eqref{eq:gradE=hgradH}.  To amplify this statement, consider multiplying both sides of  \eqref{eq:gradE=hgradH} by $\delta u_k$ to generate,
\begin{equation}\label{eq:first-var-disc}
\sum_{k} \frac{\partial E^N}{\partial u_k} \delta u_k  =  \sum_{k} h\, \partial_u H_k \delta u_k
\end{equation}
Equation~\eqref{eq:first-var-disc} may also be viewed as a discrete analog of \eqref{eq:first-var}.  Thus, one can derive \eqref{eq:first-var-disc} by the alternative process of discretizing \eqref{eq:first-var}. Using \eqref{eq:first-var-disc}, one can ``guess'' the result given by \eqref{eq:gradE=hgradH}.  That is, although \eqref{eq:first-var-disc} does not imply \eqref{eq:gradE=hgradH}, the result could have been guessed correctly starting from \eqref{eq:first-var}.  This is precisely the reason why Lemma~2 is a more granular expression of the connection between the first variation and the gradient of the Hamiltonian.  Consequently, if \eqref{eq:gradE=hgradH} were to have been guessed by discretizing \eqref{eq:first-var}, then the problems with the non-Hamiltonian implementation of a classical direct method could have been anticipated (as outlined in Section VI) and the historical chasm between direct and indirect methods might have been averted.  Unfortunately, once the direct-indirect narrative took hold, the ``corporate knowledge'' got passed on through the generations and frequently repeated and cited in subsequent papers as ``fact.''  As noted previously, \textbf{\emph{a re-analysis of some well-known narratives might be warranted in light of the arguments presented in this paper and the emerging new ideas on optimization itself\cite{rossJCAM-1}}}.

As a matter of completeness, we note that a special case of \eqref{eq:first-var-disc} is implied in the textbook by Bryson and Ho\cite{brysonHo}, page~45, Eq.~(2.2.9). This special case corresponding to a ``multi-stage decision process'' governed by the difference equation, $x(k+1) = f(x(k), u(k))$ is equivalent to \eqref{eq:FE4xdot=f} with an Euler step size of $h=1$. Substituting $h=1$ in \eqref{eq:gradE=hgradH}, we get,
\begin{equation}\label{eq:gradE=gradH}
\frac{\partial E^N}{\partial u_k} = \frac{\partial H}{\partial u_k}
\end{equation}
Equation \eqref{eq:gradE=gradH} is noted in passing in Ref.~\cite{brysonHo} (see page~45).
It is apparent that the ramifications of a ``missing'' $h$ in \eqref{eq:gradE=gradH} is nontrivial in the context that \eqref{eq:gradE=hgradH} is able to explain the myths discussed in Section~VI. Nonetheless, because the special case of $h=1$ renders the two gradients in \eqref{eq:gradE=gradH} equal to one another, it takes on an extremely important role as a foundational equation for both back propagation\cite{dreyfusBP} in Deep Learning and automatic differentiation in the so-called ``reverse mode''\cite{autoDiff}.

\section{Conclusions}

Historically, trajectory optimization methods have been classified as either direct or indirect. Under mild assumptions, a direct shooting method is mathematically equivalent to an indirect method up to certain first-order conditions.  The covector mapping principle provides a more general mathematical framework for equating a direct method to some indirect method. \emph{Because the equivalence is fundamental, the purported differences between direct and indirect methods must be relegated to a historical footnote}.

The mathematical equivalence between a direct method and some indirect method does not necessarily translate to computational equivalence if a direct method is implemented without any consideration of the Hamiltonian structure inherent in a trajectory optimization problem. In fact, the mathematical results presented in this paper explain how and why standard nonlinear programming implementations of direct methods can fail in terms of convergence and accuracy. Arguably, this is part of the reason why nonlinear programming methods for solving optimal control problems might indeed be hard. A remedy for such failures is conceptually simple: implement a direct method in a manner that incorporates the results generated by its corresponding indirect elements.  Despite the simplicity of the preceding statement, its practical realization involves a production of a double-infinite sequence of primal vectors that must be coordinated with their dual counterparts such that they converge separately with respect to their discretation and algorithmic maps. \textbf{\emph{In this context, a vast swath of research area in trajectory optimization remains unexplored and is a wide open area of research}}.

The signature trajectory optimization problem considered in this paper is also a fundamental problem in deep learning (sans the regularization term).  If a step size is used as an additional learning parameter, then Theorem~\ref{theorem:1} (and Hamiltonian programming methods) can be exploited for enhancing the accuracy, convergence and computational speed of the learning algorithms.

In recent years, there has been a substantial growth in the number of ostensibly  new trajectory optimization methods and software. A large number of these tools simply rely on patching discretization/integration methods to nonlinear programming (NLP) solvers.  Typically, the NLP solvers are third-party routines that may not allow an insertion of Hamiltonian programming principles.  The purported success achieved by such patched methods is more of a testament to the sophistication of the NLP solvers rather than a new approach to optimizing trajectories.  When the patched tools are unable to solve certain problems, it may not be that these problems are computationally hard; rather, it may be more likely a result of non-Hamiltonian programming.  Hamiltonian programming techniques are able to overcome many of the reported difficulties in solving previously-troublesome optimal control problems. The current challenges in trajectory optimization are in entirely new and unexplored areas of research in theory and computation.





\end{document}